\documentclass[12pt]{amsart}

\usepackage{graphics}
\usepackage{graphicx}

\usepackage{amsmath, amsfonts, amsthm, amssymb, color}
\usepackage{hyperref}

\usepackage{geometry}
\usepackage{enumerate}

\usepackage{subcaption}

        \definecolor{pink}{rgb}{1,0,1}

\newtheorem{theorem}{Theorem}
\newtheorem{prop}[theorem]{Proposition}

\newtheorem{cor}[theorem]{Corollary}

\newtheorem{notation}[theorem]{Notation} 

\theoremstyle{remark}
\newtheorem{remark}{Remark}
\theoremstyle{definition}
\newtheorem{defn}[theorem]{Definition}

\numberwithin{equation}{section}

\newcommand{\pa}{\partial}

\newcommand{\cL}{{\mathcal{L}}}

\newcommand{\R}{\mathbb{R}}

\newcommand{\Z}{\mathbb{Z}}

\DeclareMathOperator{\Tr}{Tr}

\title{The Neumann isospectral problem for trapezoids} 

\author[Hamid Hezari]{Hamid Hezari}\address{Department of
Mathematics, 510J Rowland Hall -
University of California, Irvine, CA 92697-3875.}
\email{hezari@math.uci.edu}

\author[Zhiqin Lu]{Zhiqin Lu}\address{Department of
Mathematics,410D Rowland Hall - University of
California, Irvine, CA 92697-3875.} \email{zlu@uci.edu}

\author[Julie Rowlett]{Julie Rowlett} \address{Mathematics
Department, Chalmers University and the
University of Gothenburg, 41296, Gothenburg Sweden}
\email{julie.rowlett@chalmers.se}

\keywords{isospectral; trapezoid; polygons; heat
invariants; wave invariants, diffraction,
inverse spectral problems. MSC primary 58C40, secondary 35P99.}

\begin{document}

\begin{abstract}

We show that trapezoids with identical Neumann spectra are congruent up to rigid
motions of the plane. The proof is
based on heat trace invariants and some new wave trace
invariants associated to certain  
diffractive billiard trajectories.  We use the method of reflections to express the Dirichlet and Neumann wave kernels in terms of the wave kernel of the \textit{double polygon}. Using Hillairet's 
trace formulas for isolated diffractive geodesics and one-parameter families of regular geodesics with \textit{geometrically diffractive} boundaries for Euclidean surfaces with conic singularities \cite{Hi}, we obtain the new wave trace invariants for trapezoids.  To handle the reflected term, we use another result of \cite{Hi}, which gives an FIO representation for the Cheeger-Taylor parametrix \cite{ChTa1, ChTa2} of the wave propagator near diffractive geodesics. The reason we can only treat the Neumann case is that the wave trace is ``more singular" for the Neumann case compared to the Dirichlet case. This is a new observation which is interesting on its own. 
\end{abstract}

\maketitle



\section{Introduction}

Our main result is the following: 
\begin{theorem} \label{th:main} Let $T_1$ and $T_2$ be two
trapezoidal domains in $\R^2$. Then if
the spectra of the Euclidean Laplacian with
Neumann boundary conditions coincide for
$T_1$ and $T_2$, the trapezoids are congruent, that is
equivalent up to rigid motions of the
plane.
\end{theorem} 
\begin{figure} \begin{center}
\includegraphics[width=200pt]{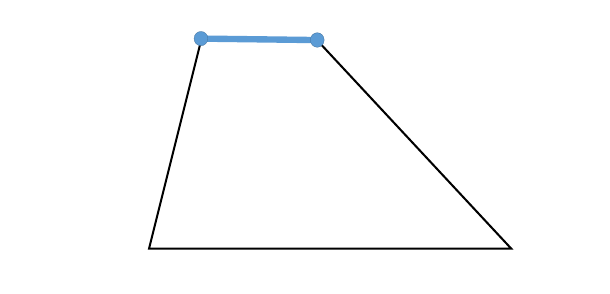} \caption{The orbit $\gamma_b$, which bounces between two diffractive conic
singularities, contributes a singularity at
$t=2b$ of order $(-\frac{1}{2})^+$ to the wave trace of the
\textit{double of the trapezoid} (as an
ESCS) if there are no other orbits of the same length. Hence $2b$ is a spectral invariant if both Dirichlet and Neumann spectra are known. In fact, as we prove, $2b$ is a spectral invariant for the Neumann spectrum.} \label{b}
\end{center} \end{figure}

Our proof relies on heat trace invariants and also some new wave trace
invariants associated to some
diffractive billiard trajectories.  We first use the method of reflections to express the Dirichlet and Neumann wave kernels in terms of the wave kernel of the \textit{double of the trapezoid} which can be realized as a Euclidean surface with conical singularities (ESCS). We then obtain new wave trace invariants using
two results of Hillairet \cite{Hi} for ESCS.  The first is a parametrix construction of the wave propagator near diffractive geodesics as an FIO, which we will use for the reflected term. Such paramatrices were found by Cheeger-Taylor \cite{ChTa1, ChTa2}, however expressing them in the language of FIOs was first done by Hillairet in \cite{Hi}. For the non-reflected term we use trace formulas of \cite{Hi} associated to isolated diffractive geodesics and to one-parameter families of regular geodesics with \textit{geometrically diffractive} boundaries. More precisely we apply the trace formulas of \cite{Hi} to the diffractive bouncing ball orbit associated to the top edge of the trapezoid (Figure \ref{b}), and to the one-parameter family of bouncing ball orbits associated to the altitudes of the trapezoid (Figure \ref{h}). The lengths of these orbits and the principal terms of the singularity expansions of the Neumann wave trace at these lengths provide new spectral invariants for the trapezoid.  Together with the well known heat trace invariants, these can be used to prove spectral uniqueness of a trapezoid amongst all trapezoids.  

The reason we can only treat the Neumann case is that in some sense the wave trace is ``more singular" for the Neumann case when compared to the Dirichlet case. This is a new feature which is of independent interest. In fact the Neumann wave trace has a larger singularity at $2b$ (See Figure \ref{b}) than the Dirichlet wave trace. It would be interesting to study the singularity of the Dirichlet wave trace at $2b$ (if singular at all), but since we do not require this for our inverse result for the Neumann boundary condition, we refrain from exploring this question here.  In a future work we shall study the isospectral problem in the Dirichlet case.  


\begin{figure} \begin{center}
\includegraphics[width=300pt]{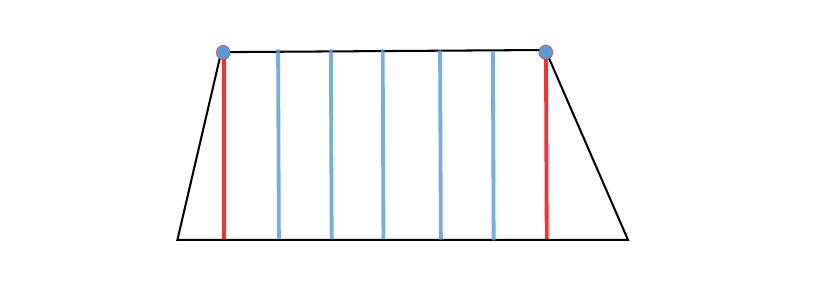} \caption{The blue
lines show the interior of a one-parameter family of
non-diffractive bouncing ball orbits of
length $2h$. The red periodic orbits, which are both
\textit{geometrically diffractive}, are the
boundaries of the blue family. The family, which we call $\gamma_h$, contributes a
singularity at $t=2h$ of order $1^+$ to
both Dirichlet and Neumann wave traces if there are no other periodic orbits of length
$2h$ with the same order of
singularity.} \label{h} \end{center} \end{figure}

\section{Background} 
The isospectral problem is: if two Riemannian manifolds are
isospectral, then are they isometric?
For a Riemannian manifold $(M, g)$ the spectrum in question is
for the Laplace operator
$$\Delta = - \sum_{i,j=1} ^n \frac{1}{\sqrt{ \det(g)}} \ \pa_i
g^{ij} \sqrt{ \det (g)} \ \pa_j.$$
The answer in this generality is \em no, \em and was proven by
Milnor in 1964 \cite{Mi}. He
used a construction of Witt \cite{Wi} of two self-dual
lattices $L_1$ and $L_2$ in $\R^{16}$
such that no rotation of $\R^{16}$ maps one to the other, but
such that the spectra of the
Riemannian manifolds $\R^{16}/L_i$ are identical for $i=1$,
$2$. Around the same time, M. Kac
wrote a popular article \cite{Ka}, ``Can one hear the shape of
a drum?'' He popularized the
isospectral problem for planar domains. Although this may seem
like an easier setting, it turned
out to be quite difficult to prove that the answer is in
general negative.

For a bounded domain $\Omega$ in $\R^2$, we consider the
Euclidean Laplacian $\Delta$ with
Dirichlet or Neumann boundary conditions,
\begin{equation} \label{lap} \Delta u(x,y) := - \frac{\pa^2
u}{\pa x^2} - \frac{\pa^2 u }{\pa y^2}
= \lambda u, \quad Bu = 0,
\end{equation} where $Bu=u|_{\partial \Omega}$ when $B=D$, and $Bu=\partial_\nu u|_{\partial \Omega}$ when $B=N$. 
For both boundary conditions, the eigenvalues, which depend on $B$, form a discrete subset of $\R$ of the form $0 \leq
\lambda_1 < \lambda_2 \leq \lambda_3
\leq \dots$. In the Dirichlet case, 
the spectrum is in bijection with the resonant frequencies a drum
would produce if $\Omega$ were its
drumhead. With a perfect ear one could hear all these
frequencies and therefore know the spectrum.
This is the origin of the title of Kac's paper \cite{Ka}.

Gordon, Webb, and Wolpert answered Kac's question in the
negative \cite{GoWeWo,GoWeWo1}, based on
Buser's work \cite{Bu}. The isospectral problem for surfaces
was previously demonstrated to
have a negative answer by \cite{Vi}. Buser's method relied on
a pasting procedure for pairs of
surfaces. In \cite{GoWeWo1}, they determined how to suitably
``fold'' two such curved surfaces to
create isospectral non-isometric planar domains. This general
idea of folding paper was later
presented in an accessible style by Chapman \cite{Chap}.

On the other hand, in some cases the isospectral problem has a
positive answer. If one considers
for example triangular domains in the plane, then if two such
domains are isospectral, the
triangles are congruent. The first proof of this fact is
contained in the doctoral thesis of C.
Durso \cite{Du}. She used the fact that the heat trace
implies that the area and perimeter are
spectral invariants, so any two triangles which are isospectral
must have the same area and
perimeter. To complete the proof, she used the wave trace and
demonstrated that the length of the
shortest closed geodesic in a triangular domain is also a
spectral invariant. More recently
Grieser and Maronna \cite{GrMa} realized that if one used an
additional spectral invariant from
the heat trace, then this together with the area and perimeter
uniquely determine the triangle.
That is a much simpler proof. Other types of domains which are
known to be spectrally determined
are analytic planar domains with reflective symmetries; see the works of Colin de Verdi\`ere \cite{CdV1, CdV2} and Zelditch \cite{Ze2, Ze}.

After triangles, one is naturally interested to know whether
the same result may hold for
quadrilaterals. For rectangles, this is a straightforward
exercise to prove that if two rectangles
are isospectral, then they are congruent. In fact, one only
requires the first two eigenvalues to
prove this fact. For parallelograms, it is also a
straightforward argument using the first three
heat trace invariants as in \cite{LuRo}. Of course the next
natural generalization is to
trapezoids. In this case, one can rather easily prove that the
geometric information which can be
extracted from the heat trace is insufficient to prove that
isospectral trapezoids are congruent.
It is therefore necessary to use the wave trace in the spirit
of \cite{Du}, which is a much
more delicate matter.
\\

In \S \ref{SpectralInvariants}, we review heat trace invariants associated to polygons and we also discover some new wave trace invariants for trapezoids. In \S \ref{ProofsOfProps}, we prove the main propositions on wave trace invariants using \cite{Hi}.  Finally, in \S \ref{ProofOfMain} we prove our inverse spectral result.


 \section{Spectral invariants of trapezoids}\label{SpectralInvariants}

\begin{defn}
A \em trapezoid \em is a convex quadrilateral which has two
parallel sides of lengths $b$ and $B$
with $B\geq b$. The side of length $B$ is called the \em base.
\em The two angles $\alpha,\beta$
adjacent to the base are called base angles. The angles at the
base satisfy
\[
0<\beta\leq\alpha\leq \frac\pi 2.  
\]
The other two sides of the trapezoid are known as \em legs \em
of lengths $\ell$ and $\ell'$,
respectively. If $\ell=\ell'$, then we say the trapezoid is \em
isosceles. \em The distance
between two parallel sides is called the \em height. \em
\end{defn}

\begin{figure} 
\begin{center}\includegraphics[width=170pt]{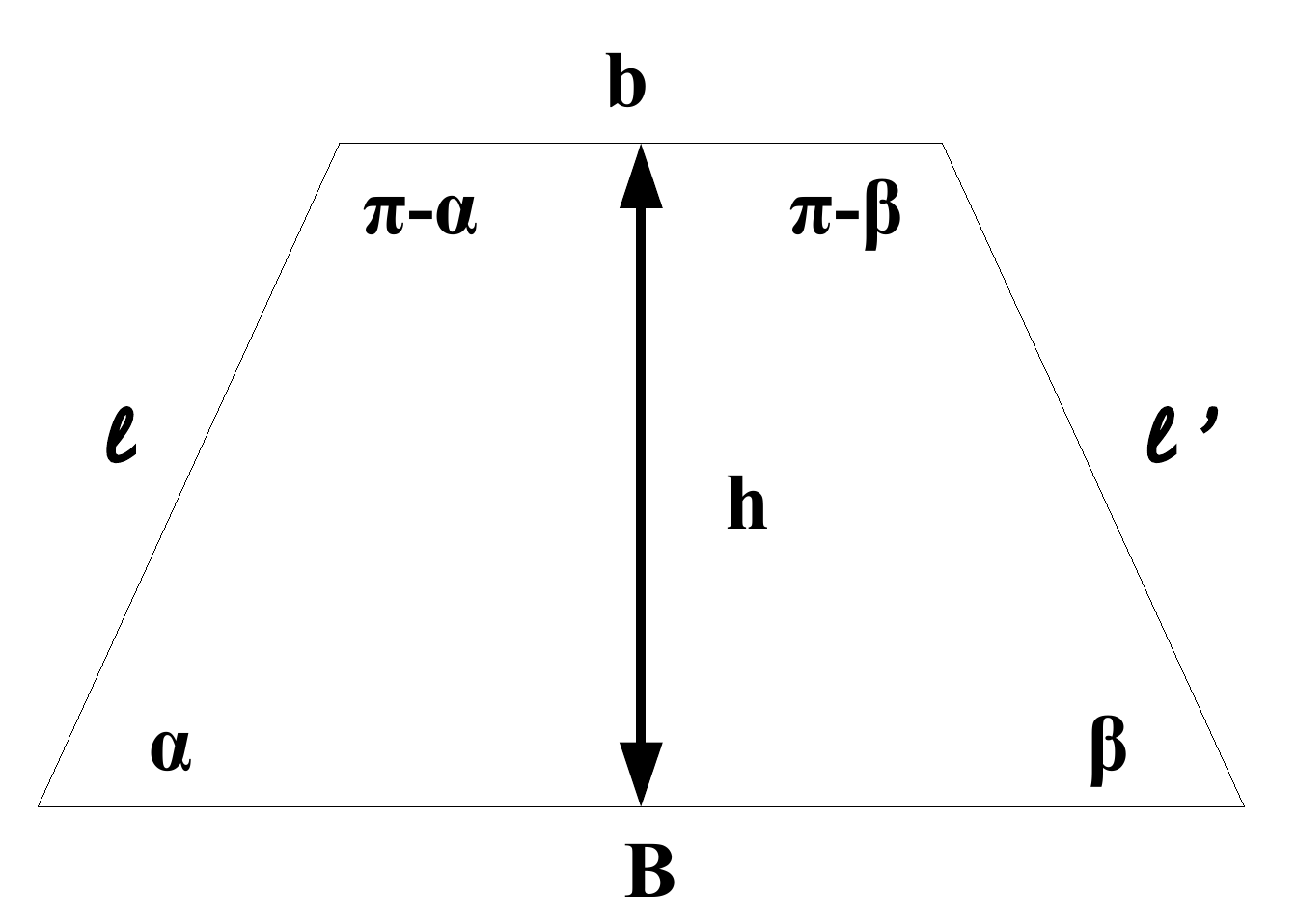}
\end{center} \caption{Parameters of a trapezoid} \label{trap}
\end{figure}

Any quantity which is uniquely determined by the spectrum is
known as a \em spectral invariant.  \em

\begin{notation}  \label{notation} If we are considering different boundary conditions on a domain $\Omega$, we shall use the notation $\Delta^B_\Omega$ to indicate the boundary condition $B$.  If we are considering a compact Riemannian manifold $M$ without boundary we shall use $\Delta_M$.  We shall use these notations when $\Omega$ is a polygon, and when $M$ is the double of a polygon as a compact Euclidean surface with conic singularities. 
\end{notation} 
 
 \subsection{The heat trace invariants} 
The heat trace
 $$\Tr e^{-\Delta^B_\Omega t} = \sum_{k \geq 1} e^{- \lambda_k(\Delta_\Omega^B) t},$$
is a spectral
invariant. It is an analytic
function for $\Re(t) > 0$ and has a singularity at $t=0$. It is
well known in this setting (see
\cite{Ka,McSi,vdBeSr,LuRo}) that the heat trace on a polygonal
domain $P$ admits an asymptotic
expansion 
\footnote{In fact in \cite{vdBeSr}, this is proved only for the Dirichlet Laplacian. That a similar asymptotic is valid for the Neumann case follows easily from the Dirichlet case and  the works of \cite{Ko, Fu} on heat trace asymptotics on ESCS; see also Remark \ref{HeatTraceNeumann}.} 
as $t \downarrow 0$,
\begin{equation} \label{HeatTrace}\Tr e^{-\Delta_P^B t} \sim \frac{|P|}{4 \pi t} +(-1)^{s(B)} \frac{| \pa
P|}{8 \sqrt{\pi t}} +
\sum_{k=1} ^n \frac{\pi^2 - \theta_k ^2}{24 \pi \theta_k} +
O(e^{-\frac{c}{t}}), \quad t \downarrow 0, \end{equation} where $c>0$, and $s(B)=1$ when $B=D$, and $S(B)=0$ if $B=N$.  Above, $|P|$ and $|\pa P|$ denote respectively the
area and perimeter of the domain
$P$, and $\theta_k$ are the interior angles. Since the
angles of a trapezoid are $\alpha$,
$\pi - \alpha$, $\beta$, and $\pi - \beta$, we therefore have
the following:  

\begin{prop} \label{AngleInvariant} For a trapezoidal domain, the area $A=|P|$, perimeter $L=|\pa
P|$, and the angle invariant
\[
q = q_{\alpha, \beta}=
\frac{1}{\alpha(\pi-\alpha)}+\frac{1}{\beta(\pi-\beta)},
\] are spectral invariants. 
\end{prop} 

\begin{remark} Note that by the definition of a trapezoid, 
$$q \geq \frac{8}{\pi^2},$$
and equality holds if and only if the trapezoid is actually a
rectangle.
\end{remark}

\begin{remark} One can show that these quantities $A$, $L$, and
$q$ are insufficient to determine
a trapezoid. In other words, considering any trapezoid $T$, up
to congruence via rigid motions of
the plane, there are infinitely many different trapezoids which have the same $A$, $L$, and $q$.
\end{remark}

\begin{remark} The remainder term in the heat trace decays
exponentially as $t \downarrow 0$, as shown
by \cite{vdBeSr, Ko}. It is therefore not feasible to extract further
geometric information from the
heat trace so we shall turn to a more subtle spectral
invariant: the wave trace. \end{remark}


\subsection{The wave trace invariants}
The wave trace is the trace of the wave propagator, also known as
the wave group, and is formally
$$w(t):=\Tr e^{i t\sqrt{\Delta} } = \sum_{k \geq 1} e^{it
\sqrt{\lambda_k} }.$$
This is purely formal, since the wave trace is only
well-defined when paired with a Schwartz class
test function; it is a tempered distribution by an easy
application of Weyl's law. It is defined
in more general settings such as compact Riemannian manifolds
without boundary as well as with boundary and with 
various boundary conditions. Duistermaat-Guillemin \cite{DuGu}
showed that in the case of compact
Riemannian manifolds without boundary the singular support of
the wave trace is contained in $\{ 0
\} \cup \pm \cL$, where $\cL$ is the set of lengths of closed
geodesics. They also found the
principal term in the singularity expansion when the orbit is
single and non-degenerate.
Guillemin-Melrose \cite{GuMe} studied this problem in the
presence of a smooth boundary and considered the 
Dirichlet, Neumman, as well as more general Robin boundary conditions. They
showed that in all cases 
$$ \text{SingSupp}\,w(t) \subset \{ 0 \} \cup \pm \cL ,$$
where$$\cL=\{ \text{lengths of generalized broken periodic
geodesics}\}.$$ Note that in this case the
length spectrum $\cL$ contains the lengths of all periodic
billiard trajectories hitting the
boundary transversally, as well as the lengths of ghost orbits
and the boundary itself when
trajectories become tangent to the boundary at some point.
Hence in a smooth convex planar domain
only the lengths of transversal billiard trajectories and the
boundary (and its multiples)
contribute to the length spectrum. Some experts conjecture that
the above containment for
$\text{SingSupp}\, w(t)$ is in fact an equality, but this has
neither been proven nor have any
counter-examples been discovered. The containment is an
equality on compact manifolds with
negative curvature \cite{DuGu}.

\subsubsection{The length spectrum of polygonal domains.}
{A polygonal domain is a planar domain whose boundary is a Euclidean polygon.  In the study of the length spectrum on such domains, the terminology polygonal table is often used due to the interpretation of a polygonal domain as the top surface of a billiard table and the identification of geodesic trajectories with billiard trajectories.}   Propagation of singularities of the wave
operator in polygonal tables or in general on manifolds with
corners or with conical singularities
are more difficult to study because of the diffraction
phenomena that takes place at the conical
singularities. Roughly speaking, when a geodesic that carries a
singularity of the wave hits a
conical singularity, it can reflect in all possible directions.
There is a huge literature on the
subject of diffraction, and for the sake of brevity we only list
the most relevant ones for our purposes: Keller
\cite{Ke}, Sommerfeld \cite{So}, Friedlander \cite{Fr1, Fr2},
Cheegar-Taylor \cite{ChTa1, ChTa2},
Melrose-Wunsch\cite{MeWu}. There has also been a lot of
research on the contribution of
diffractive geodesics to the wave trace; see Friedlander
\cite{Fr1}, Durso \cite{Du}, Wunsch
\cite{Wu}, Hillairet \cite{Hi2002, Hi}, Ford-Wunsch
\cite{FoWu}, and more recently
Hassell-Ford-Hillairet \cite{FoHaHi}. In the physics literature, we also note the works of 
Bogomolny-Pavloff-Schmit
\cite{BoPaSc} and Pavloff-Schmit \cite{PaSc}.

\begin{figure} 
\includegraphics[width=370pt]{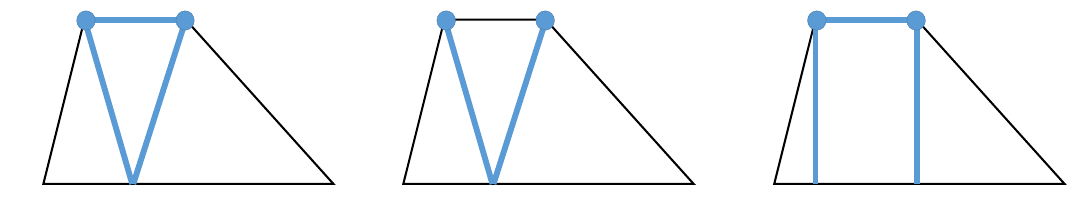}
\caption{Some diffractive periodic orbits of a
trapezoid.} \label{SomeOtherOrbits} \end{figure}

A standard technique used in studying the wave trace for polygonal tables is
to double the polygon along its
edges to obtain a compact \textit{Euclidean surface with conical 
singularities}, or ESCS, as commonly abbreviated in the literature.  A compact $n$-dimensional ESCS is a compact manifold with finitely many conical 
singularities which is locally isometric to $\R^n$ away from the conical points, 
and near conical points it is isometric
to a neighborhood of the vertex of a \textit{Euclidean cone} $C_\alpha$.

Let $P$ be a polygon, and let $P\rq{}$ be a
copy of $P$, disjoint from $P$, and $\mathcal R: P \to P\rq{}$
be the identity map. We use $2P$
for $(P \cup P\rq{}) / \sim$ where we have identified the
points of $\partial P$ and $\partial
P\rq{}$ under the map $\mathcal R$.  There is a canonical extension of $\mathcal R$ 
to an involution $\mathcal R:2P \to
2P$. The surface $2P$ is smooth everywhere except at the
vertices which are isolated conical singularities.  We note that
the cone angles are doubled under this procedure, meaning that the cone angle in the surface is twice the interior angle at the corresponding vertex in the polygon.  The Laplace operator on the ESCS arising from the doubled polygon $2P$ has many self-adjoint extensions.  We shall only consider the Friedrichs extension of the Laplacian on $2P$, which we denote simply by $\Delta_{2P}$.  

It was
proved by Hillairet \cite{Hi2002}
(and in a more general setting by Wunsch \cite{Wu}) that
$$ \text{SingSupp} \Tr e^{i t\sqrt{\Delta_{2P}}} \subset \{0\}
\cup \mathcal \pm \mathcal L_{2P}$$To describe $\mathcal
L_{2P}$ precisely, we first need to describe the geodesics on an  
ESCS, and
to do so we need some definitions. The conic points
are separated into two groups. {A conic point is called
\textit{non-diffractive} if its angle is equal to $\frac{2\pi}{N}$ for some positive integer $N$,} otherwise
it is called \textit{diffractive}.
We also use the same terminology for polygons except that
non-diffractive angles are of the form 
$\frac{\pi}{N}$. For example, if a trapezoid $T$ is not a
rectangle, then {the top vertex with
angle $\pi-\beta$ is diffractive, while the 
bottom two could be either diffractive or non-diffractive, and the top vertex with angle $\pi - \alpha$ is non-diffractive if and only if $\alpha=\frac{\pi}{2}$.}  

When a geodesic in a ESCS hits a non-diffractive conical singularity with cone angle $\frac{2 \pi}{N}$, it continues on a straight line in the cone $C_{2 \pi}$ (which is isomorphic to $\R^2$), as an $N$-fold covering space of $C_{2\pi/N}$. Hence if the incoming angle of a geodesic is $\theta_{in}$, its outgoing angle $\theta_{out}$ is $\Pi_N( \pi+ \theta_{in})$ where $\Pi_N: C_{2 \pi} \to    C_{2\pi/N}$ is the natural covering map.  In contrast, when a geodesic hits a diffractive conic point, it reflects according to \textit{Keller's democratic law of diffraction}, meaning that it reflects in all possible directions, and we  call $\theta_{out}-\theta_{in}$ the \textit{angle of diffraction}.  A geodesic is called diffractive if goes through at least one diffractive singularity (see Figure \ref{SomeOtherOrbits}). {A geodesic \textit{geometrically diffracts} at a diffractive conic point with angle $\alpha$} if it is the limit of a family of non-diffractive geodesics (see Figure \ref{h}), which happens when $\theta_{out}-\theta_{in} = \pm \pi \; \text{mod} \; \alpha \Z$.  All the above are defined similarly on polygons except that geodesics reflect on the edges according to Snell's law.   The following wavefront relations hold (\cite{ChTa1, ChTa2}) for the integral kernels of the propagators $e^{i t\sqrt{\Delta_{2P}}}$ and $e^{i t\sqrt{\Delta^B_{P}}}$ 

\begin{equation} \label{WF2P}
\text{{WF}}^{\prime} e^{i t\sqrt{\Delta_{2P}}} \subset \left \{ (t, \tau, x, \xi, y, \eta ) \in T^*(\R \times 2P \times 2P); \, \tau=|\xi|,  \,   \Phi_{2P}^t(x, \xi)=(y,\eta) \right \},
\end{equation} 

\begin{equation} \label{WFDN}
\text{{WF}}^{\prime} e^{i t\sqrt{\Delta^B_{P}}} \subset \left \{ (t, \tau, x, \xi, y, \eta ) \in T^*(\R \times P \times P); \, \tau=|\xi|,  \,   \Phi_P^t(x, \xi)=(y,\eta) \right \},
\end{equation} 
where $\Phi_{2P}$ and $\Phi_P$ are geodesics flows on $2P$ and $P$, where geodesics are defined above.   Consequently, $\mathcal L_{2P}$ and $\mathcal L_P$ are defined to be the lengths of closed geodesics, where geodesics follow the above rules of diffractions. 

\begin{remark}
We note that $\mathcal L_{2P} \subset \mathcal L_{P}$ but they are not necessarily equal. For example when $P$ is a tall trapezoid, an easy observation shows that the length of the orthic triangle (see Figure \ref{orthic}) is in $\mathcal L_{P}$ but it is not in $\mathcal L_{2P}$.
\end{remark}

\subsubsection{Singularities of wave trace on polygons.} Since the involution $\mathcal R$ commutes with $\Delta_{2P}$, there is an orthonormal basis (ONB) consisting of eigenfunctions of both operators, $\mathcal R$ and $\Delta_{2P}$. The eigenvalues of $\mathcal R$ are $\pm 1$, and hence the joint eigenfunctions of $ \mathcal R$ and $\Delta_{2P}$ are even and odd eigenfunctions  of $\Delta_{2P}$ with respect to $\mathcal R$. The even eigenfunctions of $\Delta_{2P}$ correspond to the eigenfunctions of the Neumann Laplacian on $P$, $\Delta^N_P$, and the odd eigenfunctions correspond to the ones of Dirichlet Laplacian, $\Delta^D_P$.  It is now clear that counting multiplicities we have 
$$ \text{Spec} \Delta_{2P}= \text{Spec} \Delta^D_{P} \cup \text{Spec} \Delta^N_{P},$$
and therefore 
\begin{equation} \label{DNtrace}
\Tr e^{i t\sqrt{\Delta_{2P}}}=\Tr e^{i t\sqrt{\Delta^D_{P}}}+\Tr e^{i t\sqrt{\Delta^N_{P}}}.
\end{equation}
This in particular shows that 
$$\text{SingSupp}\Tr e^{i t\sqrt{\Delta_{2P}}} \subset \text{SingSupp}\Tr e^{i t\sqrt{\Delta^D_{P}}} \bigcup \text{SingSupp}\Tr e^{i t\sqrt{\Delta^N_{P}}} .$$
\begin{remark}Again, this inclusion is not an equality because for example the length of the orthic triangle in a tall trapezoid belongs to the singular support of both traces on the right hand side but it does not belong to the singular support of $\Tr e^{i t\sqrt{\Delta_{2P}}}$. In fact what happens in this case is that the singularities of Dirichlet and Neumann wave traces (at the length of orthic triangle) cancel each other on the right hand side of \ref{DNtrace} (see Proposition \ref{OrthicSingularity}). \end{remark}

\begin{remark} \label{HeatTraceNeumann}
We point out that a similar relationship holds between the Dirichlet and Neumann heat traces and the heat trace of $2P$ as a ESCS. More precisely,
$$ \Tr e^{ -t{\Delta_{2P}}}=\Tr e^{ -t{\Delta^D_{P}}}+\Tr e^{ -t{\Delta^N_{P}}}.$$ The asymptotic expansion (\ref{HeatTrace}) was proved in \cite{vdBeSr} for the Dirichlet heat trace. We have not found a reference in literature stating the asymptotic \ref{HeatTrace} for the Neumann heat trace, however it follows immediately from the above identity and the asymptotic expansion $$\Tr e^{ -t{\Delta_{2P}}} = \frac{|P|}{2 \pi t}  +
\sum_{k=1} ^n \frac{\pi^2 - \theta_k ^2}{12 \pi \theta_k} +
O(e^{-\frac{c}{t}}), \qquad t \to 0^+, $$ proved by Kokotov (see Theorem 1 of \cite{Ko}) and Fursaev \cite{Fu}. 
\end{remark}

\begin{figure}\begin{center}
\includegraphics[width=200pt]{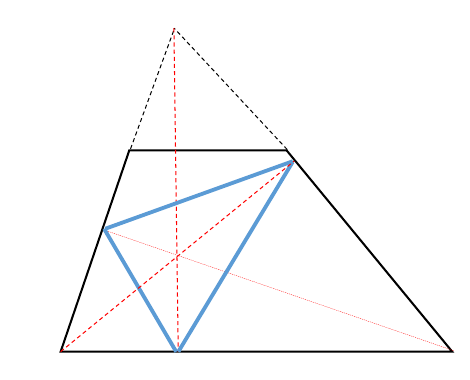} \caption{The
blue orbit corresponds to the \textit{orthic triangle}, whose
vertices are the feet of the heights
of the triangle that extends a trapezoid. It does not exists if
the trapezoid is too short or if
it is \textit{acute}: $\alpha+\beta \leq \frac{\pi}{2}$. Its
linearized Poincare map is $-I$,
hence it is a non-degenerate orbit, and by a result of
Guillemin-Melrose \cite{GuMe} it contributes a
singularity of order $\frac{1}{2}^+ $ to both Dirichlet and Neumann wave traces if
there are no other periodic orbits of
the same length with the same order of singularity. } \label{orthic}
\end{center} \end{figure} 

A common way to measure the singularity of a tempered distribution is to study the decay and growth properties of its local Fourier transform (smoothed resolvent). The following propositions are crucial to proving our inverse problems for trapezoids.

\begin{prop} \label{height}
Let $T$ be a trapezoid that is not a rectangle.  Suppose there are no other closed geodesics in $T$ of length $2h$, or arbitrarily close to $2h$, other than the one-parameter family in Figure \ref{h}. Let $\hat \rho(t) \in C^\infty_0(\R)$ be a cutoff function supported near $t=2h$ whose support  does not contain any lengths in $\{0\} \cup \pm \mathcal L_T$ other than $2h$. Then as $k \to +\infty$
$$ \int \hat \rho(t) e^{-ikt}\, \text{Tr}\, e^{i t\sqrt{\Delta^B_{T}}} dt = \frac{e^{i\pi/4}e^{-2ihk} }{\sqrt{4 \pi h}} \hat \rho(2h) A(R) k^{\frac{1}{2}} + o(k^{\frac{1}{2}}),$$
where $B=D\, \text{or}\, N$, and $A(R)$ is the area of the inner rectangle of $T$.
\end{prop} 

\begin{cor} If the conditions of Proposition \ref{height} are satisfied then $2h$ and $A(R)=bh$, the area of the inner rectangle of $T$, are spectral invariants.
\end{cor}

\begin{remark}  The above proposition, with $A(R)$ replaced by $2A(R)$, was proved by Hillairet \cite{Hi} for the trace of the wave group of $\Delta_{2T}$. In fact it was proved in a more general context, namely for ESCS and for any one-parameter family of regular periodic geodesics whose boundaries are one or two geometrically diffractive geodesics. Hence it immediately applies to the double of the trapezoid in Figure \ref{h}. However, recalling (\ref{DNtrace}), it does not immediately imply anything about the asymptotics of the traces of the wave groups associated to $\Delta^D_{T}$ or $\Delta^N_{T}$. We emphasize that Hillairet's theorem implies that if both Dirichlet and Neumann spectrum are known, then $2h$ and $A(R)$ are spectral invariants.  We do not wish to make this strong assumption, but we will nonetheless show using the method of reflections and a wavefront calculation  that the Dirchlet and Neumann wave traces have an identical singularity at $t=2h$, showing that indeed Proposition \ref{height} follows from Hillairet's result. Note that this is special for the orbits in Figure \ref{h} and does not necessarily hold for other orbits. For example as we will see in Proposition \ref{TopEdge}, the orbit in Figure \ref{b} contributes a singularity at $t=2b$ to the trace of the Neumann wave group which is larger than the singularity at $t=2b$ of the trace of the Dirichlet wave group.
\end{remark}

\begin{remark} \label{order} We note that as $k \to -\infty$ we have
$$\int \hat \rho(t) e^{-ikt}\, \text{Tr}\, e^{i t\sqrt{\Delta^B_{P}}} dt =O(|k|^{-\infty}).$$ This is because by the Fourier inversion formula
$$\int \hat \rho(t) e^{-ikt}\, \text{Tr}\, e^{i t\sqrt{\Delta^B_{P}}}dt= 2\pi \text{Tr}\, \rho \left( \sqrt{\Delta^B_{P}} -k \right).$$ However since $\rho$ is rapidly decaying near infinity, and since by the Weyl's law the eigenvalues grow linearly in dimension $2$, the trace $\text{Tr}\, \rho (\sqrt{\Delta^B_{P}} -k )$ decays rapidly as $k \to -\infty$.  This together with proposition \ref{height} shows that locally near $t=2h$, the wave trace $\text{Tr}\,e^{i t\sqrt{\Delta^B_{T}}}$ belongs to the Sobolev spaces $H^{-s}(\R)$ for all $s>1$ but does not belong to $H^{-1}(\R)$,  and for this reason we say that $t=2h$ is a singularity of order $1^+$. In general if $\text{Tr}\, e^{i t\sqrt{\Delta}}$ has an isolated singularity at $t=t_0$ and if for some $\hat \rho$ supported near $t_0$ we have as $k \to +\infty$
$$ \int \hat \rho(t) e^{-ikt}\, \text{Tr}\, e^{i t\sqrt{\Delta}} dt = c e^{-ikt_0} k^{a} + o(k^{a}),$$
for some $a \in \R$ and nonzero constant $c$, then $t_0$ is a singularity of order $(a+\frac{1}{2})^+$ meaning that near $t=t_0$ the wave trace belongs to $H^{-s}(\R)$ for all $s > a+\frac{1}{2}$ but does not for $s=a+\frac{1}{2}$. 
\end{remark}

\begin{remark}
The discussion in remark \ref{order} also shows that as $k \to +\infty$ 
$$ \int \hat \rho (t) e^{-ikt} \text{Tr}\,  \cos ( t \sqrt{\Delta}) dt = \frac{1}{2} \int \hat \rho (t) e^{-ikt} \text{Tr}\, e^ {it \sqrt{\Delta}} dt + O(k^{-\infty}).$$ This shows a relation between $\Tr \cos ( t \sqrt{\Delta})$ and $\Tr e^ {it \sqrt{\Delta}}$. 
\end{remark}

The next proposition concerns the diffractive orbit $\gamma_b$ in Figure \ref{b}.

\begin{prop} \label{TopEdge}
Let $T$ be a trapezoid with $\alpha \neq \frac{\pi}{2}$ and $\beta \neq \frac{\pi}{2}$ .  Suppose there are no closed geodesics in $T$ of length $2b$, or arbitrarily close to $2b$, other than $\gamma_b$ in Figure \ref{b}. Let $\hat \rho(t) \in C^\infty_0(\R)$ be a cutoff function supported near $t=2b$ whose support  does not contain any lengths in $\{0\} \cup \pm \mathcal L_T$ other than $2b$. Then as $k \to +\infty$
$$\quad \int \hat \rho(t) e^{-ikt}\, \text{Tr}\, e^{i t\sqrt{\Delta^B_{T}}} dt = -\pi i \hat \rho(2b) e^{-2bki}   C_{\alpha, \beta}  k^{-1} + O(k^{-2}),$$ for $B=\text{Neumann}$, and
$$\quad \int \hat \rho(t) e^{-ikt}\, \text{Tr}\, e^{i t\sqrt{\Delta^B_{T}}} dt = O(k^{-2}),$$ for $B=\text{Dirichlet}$. Here the constant $C_{\alpha, \beta}$ is given by
\begin{equation}\label{C1} C_{\alpha, \beta}= \frac{ \cot(\frac{\pi^2}{2\pi-2\alpha})\cot(\frac{\pi^2}{2\pi-2\beta})}{(\pi-\alpha)(\pi-\beta)}. \end{equation}


When $\alpha= \frac{\pi}{2}$ and $\beta \neq \frac{\pi}{2}$, as $k \to \infty$ we have $$\quad \int \hat \rho(t) e^{-ikt}\, \text{Tr}\, e^{i t\sqrt{\Delta^B_{T}}} dt =  (2\pi b)^{\frac{1}{2}} \hat \rho(2b) e^{-\frac{\pi i}{4}} e^{-2bki}   C_\beta k^{-\frac{1}{2}} + O(k^{-\frac{3}{2}}),$$ for $B=\text{Neumann}$, and
$$\quad \int \hat \rho(t) e^{-ikt}\, \text{Tr}\, e^{i t\sqrt{\Delta^B_{T}}} dt = O(k^{-\frac{3}{2}}),$$ for $B=\text{Dirichlet}$. Here 
\begin{equation}\label{C2} C_{\beta}= -\frac{ \cot(\frac{\pi^2}{2\pi-2\beta})}{\pi-\beta}. \end{equation}

\end{prop}

As a quick corollary we obtain a new angle invariant. 
\begin{cor}\label{NewAngleInvariant} Let $T$ be a trapezoid such that there no orbits of length $2b$ other than the bouncing ball orbit $\gamma_b$ corresponding to the top edge. If $\beta \leq  \alpha < \frac{\pi}{2}$, then $2b$ and $C_{\alpha, \beta}$ defined by \ref{C1} are spectral invariants for the Neumann spectrum. If $\beta < \alpha = \frac{\pi}{2}$, then $2b$ and $C_\beta$ defined by \ref{C2} are spectral invariants for the Neumann spectrum.

\end{cor}

\begin{remark} Again, this proposition follows from Hillairet \cite{Hi} with required modifications to separate the Dirichlet and Neumann wave traces, which we will discuss in the proof.  
\end{remark}

\begin{remark}The following proposition might be useful when one studies the isospectral problem on trapezoids for the Dirichlet Laplacian. It concerns the wave trace contribution of the orthic orbit in Figure \ref{orthic}. Up to the principal part, it is a direct consequence of Guillemin-Melrose trace formula  \cite{GuMe} for simple and non-degenerate periodic orbits. We state the proposition without the proof because we do not use it in this paper. In the following we use $l_F$ for the length of the orthic (also called Fagnano) triangle which by the notations of Figure \ref{trap} equals $2B \sin \alpha \sin \beta$. 
\end{remark}

\begin{prop} \label{OrthicSingularity}
Let $T$ be a trapezoid with $\alpha \, \text{and}\, \beta \neq \frac{\pi}{2}$, and $ \alpha+\beta > \frac{\pi}{2}$. Suppose $T$ is tall enough that the orthic (Fagnano) triangle lies in $T$ and is non-diffractive as in Figure \ref{orthic}.  Suppose there are no other closed geodesics in $T$ of length $l_F$, other than the orthic triangle. Let $\hat \rho(t) \in C^\infty_0(\R)$ be a cutoff function supported near $t=l_F$ whose support  does not contain any lengths in $\{0\} \cup \pm \mathcal L_T$ other than $l_F$. Then as $k \to +\infty$
$$ \int \hat \rho(t) e^{-ikt}\, \text{Tr}\, e^{i t\sqrt{\Delta^B_{T}}} dt  \sim (-1)^{s_B} \,l_F\,e^{-ikl_F} \hat \rho(l_F) \sum_{j=0}^\infty {c_j}{k^{-j}},$$
where $s_B=0$ if $B=D$ and $s_B=1$ if $B=N$. The constant $c_0$ is nonzero.  Moreover, the constants $\{c_j\}_{j=0}^\infty$ depend only on $l_F$  and are independent of $B$. Hence, the invariants $\{c_j\}_{j=0}^\infty$ do not introduce any spectral invariants other than $l_F$.  
\end{prop}

\begin{cor} Under the conditions of Proposition \ref{OrthicSingularity}, $l_F=2B \sin \alpha \sin \beta$ is a spectral invariant for both Dirichlet and Neumann spectra. 

\end{cor}

\section{Proofs of Propositions \ref{height} and \ref{TopEdge}}\label{ProofsOfProps}

Let $P$ be a polygonal domain and define $2P$ and the involution map $\mathcal R: 2P \to 2P$ as in the previous section. We denote

$$ U_{2P}(t) = e^{it \sqrt{ \Delta_{2P}}}, \quad  U^D_P(t) = e^{it \sqrt{ \Delta^D_P}},  \quad U^N_P(t) = e^{it \sqrt{ \Delta^N_P}},$$ and we use 
$$U_{2P}(t, x, y), \quad U^D_{P}(t, x, y), \quad U^N_{P}(t, x, y),$$ for their integral kernels. The following proposition expresses the Dirichlet and Neumann wave kernels in terms of $U_{2P}(t, x, y)$. 

\begin{prop} For all $t \in \R$ and all $x, y \in P$:
$$U^D_{P}(t, x, y)= U_{2P}(t, x, y) - U_{2P}(t, x, \mathcal R y),
$$
$$U^N_{P}(t, x, y)=   U_{2P}(t, x, y) + U_{2P}(t, x, \mathcal R y).
$$
\end{prop} 

The proof is obvious from the expansion of $U_{2P}(t, x, y)$ in terms of an ONB of eigenfunctions of $\Delta_{2P}$  consisting of even and odd eigenfunctions with respect to $\mathcal R$.



As an immediate corollary we obtain:
\begin{cor} \label{DNtraces}
$$ \text{Tr} \, U^D_P(t) =\frac{1}{2} \left ( \text{Tr}\, U_{2P}(t) - \int _{2P} U_{2P}(t, x, \mathcal R x) dx \right) ,$$
$$ \text{Tr} \, U^N_P(t) =\frac{1}{2} \left ( \text{Tr}\, U_{2P}(t) + \int _{2P} U_{2P}(t, x, \mathcal R x) dx \right ).$$ 
\end{cor}  
We now specialize to the case of a trapezoid. To prove Propositions \ref{height}, \ref{TopEdge}, and \ref{OrthicSingularity}, we use Corollary \ref{DNtraces} to reduce the problem to studying asymptotics of tempered distributions $\text{Tr}\, U_{2P}(t)$ and $\int _{2P} U_{2P}(t, x, \mathcal R x) dx $. Theorem {2}  of Hillairet \cite{Hi} gives the asymptotics of the trace $\text{Tr}\, U_{2P}(t)$, but the term  $\int _{2P} U_{2P}(t, x, \mathcal R x) dx$ is a new ingredient which is relatively easy to study. 

\begin{proof}[Proof of Proposition \ref{height}] Let $\epsilon >0$ such that there are no lengths other than $2h$ in the interval $(2h-\epsilon, 2h+\epsilon)$. To prove Proposition \ref{height} it suffices to show that on the interval $(2h-\epsilon, 2h+\epsilon)$ we have
$$\text{WF}\, \left (\int _{2P} U_{2P}(t, x, \mathcal R x) dx \right ) =\emptyset,$$
because this would imply that 
$$ \int \hat \rho (t) e^{-ikt} \left(\int _{2P} U_{2P}(t, x, \mathcal R x) dx \right)   \,dt = O(k^{-\infty}). $$ To prove the emptiness of the above wavefront set, we just need to follow the argument as in \cite{DuGu} and write 
$$ \int _{2P} U_{2P}(t, x, \mathcal R x) dx= \Pi_* \triangle^* ( U_{2P} \circ \mathcal R),$$
where $\triangle^*$ is the pullback by the diagonal map $\triangle: \R \times 2P \to \R \times 2P \times 2P$, and $\Pi_*$ is the pushforward by the projection map $\Pi: \R \times 2P \to \R$. The same wavefront calculations as in \cite{DuGu} shows that 
$$\text{WF} \int _{2P} U_{2P}(t, x, \mathcal R x)dx $$ $$  \subset \left \{ (t, \tau); \tau>0, \exists (x, \xi) \in T^*(2P): \Phi^t_{2P}(x, \xi)=(\mathcal R x, \xi) \right \}.$$ 
Now suppose $t_0\in (2h -\epsilon, 2h+\epsilon)$, and $\Phi^{t_0}_{2P}(x, \xi)=(\mathcal R x, \xi) $ for some $(x, \xi) \in T^*(2P)$. Then the projection of the geodesic segment $\{\Phi^{t}_{2P}(x, \xi)\}_{t \in [0, t_0]}$ onto $T^*P$, under the the natural projection $\pi: 2P \to P$, is a closed geodesic in $P$ (as a billiard table) of length $t_0$. However, by assumption the only periodic orbits in $P$ of length in the interval $(2h -\epsilon, 2h +\epsilon)$ must belong to the one-parameter family in Figure \ref{h}. Hence $t_0=2h$, and the projection of $\{\Phi^{t}_{2P}(x, \xi)\}_{t \in [0, 2h]}$ onto $P$ must be a bouncing ball orbit parallel to the altitude of the trapezoid. Unfolding this onto $2P$ we get that  $\Phi^{t_0}_{2P}(x, \xi)=(x, \xi)$, and since $\Phi^{t_0}_{2P}(x, \xi)=(\mathcal R x, \xi)$, we must have $\mathcal R x =x$, or equivalently $x \in \partial P$. However we {may repeat the same argument with a point $(x, \xi)$ in the orbit which is in the interior of $P$ since the orbit is parallel to the altitude.} This gives a contradiction. 

\end{proof}

\begin{remark} The wavefront calculation above shows that in general for the distribution $\int _{2P} U_{2P}(t, x, \mathcal R x) dx$ to have nonempty wavefront set near the length of an orbit (with no other lengths nearby), it is required that the orbit lies entirely on the boundary of $P$. This is precisely what happens in Proposition \ref{TopEdge}.  
\end{remark}


\begin{proof}[Proof of Proposition \ref{TopEdge}] Theorem 2 of \cite{Hi} gives the asymptotics for $\text{Tr}\, U_{2P}(t)$ near $t=2b$, which are exactly those given in Proposition \ref{TopEdge}. Hence, by Corollary \ref{DNtraces}, to prove this proposition it suffices to show that
\begin{equation} \label{Traces} \int _\R \int _{2P} \hat \rho (t) e^{-ikt} U_{2P}(t, x, \mathcal R x) dx dt = \int _\R \int _{2P} \hat \rho (t) e^{-ikt} U_{2P}(t, x, x) dx dt+ O(k^{-\frac{n}{2}-1}), \end{equation} where $n=2$ if $\beta \leq \alpha \neq \frac{\pi}{2}$, and $n=1$ when $\beta< \alpha = \frac{\pi}{2}$. We note that in fact $n$ corresponds to the number of diffractions because there is no diffraction at the top left vertex when $\alpha = \frac{\pi}{2}$. 

To prove the proposition we use Theorem 5 of \cite{Hi} which gives a parametrix for $U_{2P}(t, x, y)$ microlocalized near a diffractive geodesic connecting a point $x_0$ near $x$ to a point $y_0$ near $y$.
 
\begin{theorem}[Hillairet] Let $P$ be a polygon and $\gamma$ be a diffractive geodesic on $2P$ of length $t_0$, with initial and terminal points $x_0$ and $y_0$ in $2P$, going through $n$ diffractions at conic points $p_1, p_2, \dots, p_n$ of angles $\alpha_1, \alpha_2, \dots, \alpha_n$, with angles of diffractions $\beta_1, \beta_2, \dots, \beta_n$. Let $(r, \theta)$ and $(R, \Theta)$ be polar coordinates centered at $p_1$ and $p_n$, chosen in such a way that the line segments $\overline{x_0p_1}$ and $\overline{p_ny_0}$ correspond to $\theta=0$ and $\Theta=0$ respectively. Then microlocally near $\gamma$,  $U_{2P}$ is an FIO, and near $(t_0, x_0, y_0)$ and away from the conic points, has a parametrix of the form
$$ \tilde {U}_{2P, \gamma}(t, x, y) =\int_{\xi >0} e^{i\xi \left (t-R_0(x)-R_1(y)- \sum^{n-1}_{j=1}L_j \right )} a_\gamma(t, x, y, \xi) d \xi,$$where $L_j=d(p_j, p_{j+1})$ and as $\xi \to +\infty$ the amplitude $a_\gamma$ has an asymptotic expansion of the form $$a_\gamma (t, x, y, \xi)\sim \sum_{m=0}^ \infty a_m(t, x, y) \xi^{-\frac{n-1}{2}-m},$$ with leading term
$$a_0(t, x, y) =(2\pi)^{(n-3)/2} e^{−(n-1)i\pi/4}  \frac{S_\gamma(x, y)}{\sqrt{L_\gamma(x,y)}}.$$ 
Here
$$S_\gamma (x, y) = S_{\alpha_1}(\beta_1- \theta(x)) S_{\alpha_2}(\beta_2) \dots S_{\alpha_{n-1}}(\beta_{n-1}) S_{\alpha_n}(\Theta(y) - \beta_n),$$ and 
$$L_\gamma(x, y)=R_0(x) L_1 L_2 \dots L_{n-1} R_1(y),$$
where
$$S_\delta (\eta)= -\frac{\sin(\frac{2\pi^2}{\delta})}{2\delta\sin(\frac{\pi}{\delta}(\pi+\eta)) \sin(\frac{\pi}{\delta}(\pi-\eta))},$$ which at $\eta=0$ simplifies to $S_\delta(0)= -\frac{1}{\delta} \cot(\frac{\pi^2}{\delta})$. 
\end{theorem} 
We now apply this theorem to $\gamma_b$ in Figure \ref{b}. First we choose the coordinates so that the top left corner of $T$ is at $C_1:=(0, 0)$, and the top right corner is at $C_2:=(b, 0)$, hence $\gamma_b$ lies on the $x_1$ axis. We then reflect $T$ about the $x_1$ axis. In particular, this would give a natural neighborhood of the interior of $\gamma$, and the involution map $\mathcal R$ becomes $\mathcal R(x_1, x_2)= \mathcal R(x_1, -x_2)$. We also choose three cutoff functions, $\chi_{C_1}$, $\chi_{C_2}$, and $\chi$ on $2T$, all invariant under $\mathcal R$.  These are chosen to satisfy:  $\chi_{C_1}+\chi_{C_2}+\chi =1$ near $\gamma_b$;  $\chi_{C_1}$ and $\chi_{C_2}$ are supported in small neighborhoods of $C_1$ and $C_2$, respectively; and $\chi$ is supported away from $C_1$ and $C_2$. By a wavefront calculation as in the proof of the previous proposition we can see that
$$\int \hat \rho (t) e^{-ikt} \Tr U_{2P}(t) \circ \mathcal R\, \chi \, dt = \int \hat \rho (t) e^{-ikt} \Tr \tilde{U}_{2P, \gamma}(t) \circ \mathcal R \, \chi \, dt+ O(k^{-\infty})$$
Newt, we substitute the parametrix given in the statement of the theorem for $\tilde {U}_{2P,\gamma}(t, x, y)$. However, an immediate observation shows that $R_1( \mathcal R x) =R_1 (x)$, and $\Theta( \mathcal R x) = \Theta(x)$. Hence the phase functions and the leading terms of the amplitudes of the oscillatory integrals $\tilde {U}_{2P,\gamma}(t, x, \mathcal R x)$ and $\tilde {U}_{2P,\gamma}(t, x, x)$ agree on Supp$\chi$. By the stationary phase lemma, as performed in the proof of Theorem 5 of \cite{Hi}, we have 
$$\int \hat \rho (t) e^{-ikt} \Tr \tilde{U}_{2P, \gamma}(t) \circ \mathcal R \, \chi \, dt=\int \hat \rho (t) e^{-ikt} \Tr \tilde{U}_{2P, \gamma}(t) \, \chi \, dt + O(k^{-\frac{n}{2}-1}).$$ 
This implies \ref{Traces} because 
$$\int \hat \rho (t) e^{-ikt} \Tr {U}_{2P }(t)  \, \chi \, dt=\int \hat \rho (t) e^{-ikt} \Tr \tilde{U}_{2P, \gamma}(t)  \, \chi \, dt + O(k^{-\infty}).$$
Near the conic points $C_1$ and $C_2$ we can use the cyclicity of the trace, as used by \cite{Du, Hi, FoHaHi} to move the support of the integrands away from the conic points and reduce to the setting above. 
\end{proof}

\section{Proof of Theorem \ref{th:main}}\label{ProofOfMain}
Our first simple observation is
\begin{prop} \label{Shortest} The length of any periodic orbit in a trapezoidal table $T$ is strictly larger than $2h$ or $2b$ unless the orbit is a bouncing ball corresponding to one of the altitudes or it is the bouncing ball between the top two vertices. \end{prop} 
\begin{proof} We note that any closed diffractive or non-diffractive geodesic in $T$ that starts from the top edge (including the corners) and is transversal (i.e. not tangent) to the top edge must be of length strictly larger than $2h$ unless it is parallel to the altitude.  Furthermore, any geodesic that touches the left and right edges (including the corners)  must be of length larger than $2b$ unless it is the bouncing ball orbit $\gamma_b$. If a geodesic touches the bottom edge and the right edge (respectively, left edge) then it must also visit the top edge or the left edge (respectively, right edge) and hence its length is larger than $2h$ or $2b$.   \end{proof}

The main theorem follows immediately by combining the following four propositions and the heat trace invariants $A$, $L$, and $q_{\alpha, \beta}$.  

\begin{prop} \label{Rectangular}  Let $T_1$ and $T_2$ be two trapezoids with the same Neumann spectra. If $T_1$ is a rectangle, then $T_2$ is a rectangle that is congruent to $T_1$.  \end{prop}

\begin{prop}\label{WaveInvariants} Let $T_1$ and $T_2$ be two non-rectangular trapezoids with the same Neumann spectra. 
\begin{itemize}
\item If $h(T_1) \leq b(T_1)$, then $h(T_1)=h(T_2)$ and $b(T_1)=b(T_2)$. 

\item If $b(T_1) < h(T_1)$ then $b(T_1)=b(T_2)$. In addition if $\alpha(T_1) \neq \frac{\pi}{2}$, then $\alpha(T_2) \neq \frac{\pi}{2}$, and $C_{\alpha(T_1), \beta(T_1)}=C_{\alpha(T_2), \beta(T_2)}$, where $C_{\alpha, \beta}$ is defined by (\ref{C1}). Moreover, if $\alpha(T_1) =\frac{\pi}{2}$, then $\alpha(T_2) =\frac{\pi}{2}$ and $C_{ \beta(T_1)}=C_{ \beta(T_2)}$ where $C_\beta$ is given by (\ref{C2}). 
\end{itemize}
\end{prop}

\begin{prop} \label{2h}
If two trapezoids have the same area $A$, perimeter $L$, height $h$, and $b$, then they are congruent up to rigid motions.  
\end{prop}

\begin{prop} \label{2b}
If two trapezoids, with $\alpha(T_1)\, \text{and} \; \alpha(T_2) \neq \frac{\pi}{2}$, have the same area $A$, angle invariant $q_{\alpha, \beta}$, the same $b$, and the same $C_{\alpha, \beta}$, then they are congruent up to rigid motions.  Moreover, if two non-rectangular trapezoids, with $\alpha(T_1)=\alpha(T_2)= \frac{\pi}{2}$, have the same area $A$, angle invariant $q_{\alpha, \beta}$, and the same $b$, then they are congruent up to rigid motions.
\end{prop}
We now give the proofs of these propositions. 

\begin{proof} [Proof of Proposition \ref{Rectangular} ] The angle invariant satisfies $$q \geq
\frac{8}{\pi^2}$$ with equality if and only if the
trapezoid is a rectangle.  {Since $T_1$ and $T_2$ are isospectral, they have the same angle invariant.  Consequently, since $T_1$ is a rectangle, $T_2$ is as well.  Furthermore, by isospectrality, $T_1$ and $T_2$ have the same area and perimeter, and these uniquely determine a rectangle up to rigid motions.}  
\end{proof}

\begin{proof} [Proof of Proposition \ref{WaveInvariants}] 
First suppose $h(T_1) < b(T_1)$. Then by Proposition \ref{Shortest}, $2h(T_1)$ is the shortest length in $\mathcal L_{T_1}$ and there are no orbits other than the one-parameter family of altitudes having the same length. Hence by Proposition \ref{height}, both Dirichlet and Neumann wave traces of $T_1$ have a singularity of order $1^+$ at $t=2h(T_1)$ where up to a constant the leading coefficient equals  $b(T_1)h(T_1)$, the area of the inner rectangle. Since $T_1$ and $T_2$ are isospectral, the same must hold for the wave trace of $T_2$. In particular, we must have $2h(T_1)=2h(T_2)$, and $b(T_1)h(T_1)=b(T_2)h(T_2)$, so that $b(T_1) = b(T_2)$.

If $b(T_1) < h(T_1)$, then again using Proposition \ref{Shortest}, $2b(T_1)$ is the shortest length in $\mathcal L_{T_1}$.  By Proposition \ref{TopEdge}, the Neumann wave trace of $T_1$ has a singularity at $t=2b(T_1)$. If the order of this singularity is $(-\frac{1}{2})^+$, then we know that $\alpha(T_1) \neq \frac{\pi}{2}$, and the same type of singularity is found in the Neumann wave trace of $T_2$, {thus $\alpha(T_2) \neq \frac{\pi}{2}$ as well. Furthermore, $2b(T_1)=2b(T_2)$, and $C_{\alpha(T_1), \beta(T_1)}=C_{\alpha(T_2), \beta(T_2)}$.  Similarly, if the order of this singularity is $0^+$, then we know that there is only one diffraction, meaning that $\alpha(T_1) = \frac{\pi}{2}$. Since the singularity in the wave trace of $T_2$ must be the same, we must have $\alpha(T_2) = \frac{\pi}{2}$, $2b(T_1)=2b(T_2)$ and $C_{\beta(T_1)}=C_{\beta(T_2)}$. }

When $h(T_1)=b(T_1)$, since there are no orbits of length $2h(T_1)=2b(T_1)$ other than $\gamma_h$ and $\gamma_b$, the singularities of the wave trace at $t=2h(T_1)$ and $t=2b(T_1)$ add. This is because in fact Propositions \ref{height} and \ref{TopEdge} are also valid microlocally near their corresponding orbits. In this case since the singularity at $t=2h(T_1)$ is larger, and it contributes to the leading singularity of the wave trace. Hence as in the first case, we have $2h(T_1)=2h(T_2)$, and $b(T_1)h(T_1)=b(T_2)h(T_2)$, thus $b(T_1) = b(T_2)$. 
\end{proof}

\begin{proof} [Proof of Proposition \ref{2h}] If $A$, $b$, and $h$ are known, then obviously $B$ can be determined. On the other hand, it is clear from Figure \ref{trap} that $$\frac{B-b}{h} =\cot \alpha+\cot \beta, \qquad \frac{L-B-b}{h}=\csc \alpha+\csc \beta.$$ Hence $\csc \alpha+\csc \beta$ and $\cot \alpha+\cot \beta$ are spectrally determined. By adding and subtracting these two invariants we arrive at
$$ \cot \frac{\alpha}{2}+\cot \frac{\beta}{2}, \qquad \tan \frac{\alpha}{2}+\tan \frac{\beta}{2},$$ as two spectrally determined quantities. Since $\cot \frac{\alpha}{2}+\cot \frac{\beta}{2}= \frac{ \tan \frac{\alpha}{2}+\tan \frac{\beta}{2}}{\tan \frac{\alpha}{2}\tan \frac{\beta}{2}}$, we can also determine $\tan \frac{\alpha}{2}\tan \frac{\beta}{2}$, which uniquely determines $\tan \frac{\alpha}{2}$ and $\tan \frac{\beta}{2}$.  Therefore $\alpha$ and $\beta$ are spectrally determined because $0< \beta \leq \alpha \leq \frac{\pi}{2}$. 

\end{proof}

\begin{proof} [Proof of Proposition  \ref{2b}]

Our plan is to show that for $\beta \leq \alpha < \frac{\pi}{2}$, the pair $\{ q_{\alpha, \beta}, C_{\alpha, \beta} \}$ determines $\alpha$ and $\beta$ uniquely. To do this we show that $C_{\alpha, \beta}$ is an increasing function on the level curves of $q_{\alpha, \beta}$ as $\alpha$ increases. 

We recall from Proposition \ref{AngleInvariant} that \begin{equation} \label{dqfa}
q_{\alpha, \beta}=\frac{1}{\alpha(\pi-\alpha)}+\frac{1}{\beta(\pi-\beta)} =
F(\alpha) + F(\beta), \quad F(\alpha)
:= \frac{1}{\alpha (\pi - \alpha)}.
\end{equation}
Under the assumptions
that $q\geq 8/\pi^2$ and $\beta \leq \alpha$, which are always valid for trapezoids, each $q$ and $\alpha$ uniquely
determine $\beta$ by
$$\beta=\beta(\alpha)=\frac\pi
2\left(1-\sqrt{1-\frac{4}{\pi^2(q-F(\alpha))}}\right).
$$

Since $\beta \leq \alpha$, the range for $\alpha$ is
{$[\alpha_0,\pi/2)$}, where $\alpha_0$ is
determined by setting $\alpha= \alpha_0 = \beta$,
\[
\alpha_0=\frac\pi 2\left(1-\sqrt{1-\frac{8}{q\pi^2}}\right).
\]
Then $F(\alpha_0)=q/2$.   Using implicit differentiation in the equation  
$$q = F(\alpha) + F(\beta),$$
we have
\begin{equation} \label{dbeta} \beta'(\alpha) = - \frac{F'(\alpha)}{F'(\beta)}, \end{equation} where 
$$ \label{dF}
 F'(\alpha)=-\frac{\pi-2\alpha}{\alpha^2(\pi-\alpha)^2}\leq 0
$$
for $\alpha\in[\alpha_0,\pi/2]$.  {The inequality is strict when $\alpha > \beta$.}  

We also recall that 
$$ C_{\alpha, \beta}= \frac{ \cot(\frac{\pi^2}{2\pi-2\alpha})\cot(\frac{\pi^2}{2\pi-2\beta})}{(\pi-\alpha)(\pi-\beta)}.$$ 
Since $\alpha, \; \beta \in (0, \frac{\pi}{2})$, both $\cot(\frac{\pi^2}{2\pi-2\alpha})$ and $\cot(\frac{\pi^2}{2\pi-2\beta})$ are negative, and $C_{\alpha, \beta} >0$.  We will show that $\frac{d}{d\alpha} \log C_{\alpha, \beta} \geq 0$, and that it is zero if and only if $\beta=\alpha=\alpha_0$. Using  the identity $\frac{1+ \cot^2(a)}{\cot a}=\frac{2}{\sin 2a}$ and (\ref{dbeta}), we get
\begin{align*}   \frac{d}{d\alpha} \log C_{\alpha, \beta}  & = \left (\frac{1}{\pi- \alpha} -\frac{\pi^2}{(\pi-\alpha)^2\sin { \frac{\pi^2}{\pi-\alpha}}} \right) + \beta ' (\alpha) \left ( \frac{1}{\pi- \beta} -\frac{\pi^2}{(\pi-\beta)^2\sin { \frac{\pi^2}{\pi-\beta}}} \right )\\ 
&  =-F'(\alpha) \left ( \frac{\alpha^2}{\pi-2 \alpha}\left ( \pi-\alpha-\frac{\pi^2}{\sin \frac{\pi^2}{\pi -\alpha}}\right) - \frac{\beta^2}{\pi-2 \beta}\left ( \pi-\beta-\frac{\pi^2}{\sin \frac{\pi^2}{\pi -\beta}}\right) \right ).
\end{align*}

We now define
$$G(\sigma)=\frac{\sigma^2}{\pi-2 \sigma}\left ( \pi-\sigma-\frac{\pi^2}{\sin \frac{\pi^2}{\pi -\sigma}}\right).$$ 
Since $F'(\alpha) <0$ {for all $\alpha > \beta$, } and $ \beta \leq \alpha $, to prove that  $\frac{d}{d\alpha} \log C_{\alpha, \beta} \geq 0$ {with equality if and only if $\beta = \alpha$}, we need to show that $G(\sigma)$ is an increasing function on $(0, \frac{\pi}{2})$. It is more convenient to change the variable by 
$$ \theta = \frac{\pi^2}{\pi -\sigma}-\pi.$$ 
{Then since $\sigma \in (0, \pi/2)$, we have $\theta \in \left(0, \pi \right).$}

Then $G$ as a function of $\theta$ becomes
$$G(\theta)= \frac{\pi^3 \theta^2}{(\pi+\theta)(\pi-\theta)} \left (\frac{1}{\sin \theta} + \frac{1}{\pi+\theta} \right ), \qquad \theta \in (0, \pi). $$ 
To show that $G(\theta)$ is increasing on $(0, \pi)$, we prove that $\frac{d}{d\theta} G(\theta) >0$ on $(0, \pi)$.  {Since $\theta \in (0, \pi)$, it is clear to see that $G(\theta) > 0$.  It therefore suffices to prove that $\frac{d}{d\theta} \log G(\theta) > 0$.} A simple calculation shows that 
\begin{align*} \frac{d}{d\theta} {\log} G(\theta) & = \frac{2}{\theta} -\frac{1}{\pi+\theta} +\frac{1}{\pi-\theta} - \frac{\frac{\cos \theta}{\sin^2 \theta}+\frac{1}{(\pi+\theta)^2}}{\frac{1}{\sin \theta}+\frac{1}{\pi+\theta}}\\
& > \frac{2}{\theta} -\frac{2}{\pi+\theta} +\frac{1}{\pi-\theta} - \frac{\frac{\cos \theta}{\sin^2 \theta}}{\frac{1}{\sin \theta}+\frac{1}{\pi+\theta}}. 
\end{align*}

Clearly if $ \theta \in [\frac{\pi}{2}, \pi)$, then this implies that 
$$  \frac{d}{d\theta} {\log} G(\theta) > \frac{2}{\theta} -\frac{2}{\pi+\theta} +\frac{1}{\pi-\theta} > \frac{8}{3 \pi} >0.$$ 
On the other hand, if $ \theta \in (0, \frac{\pi}{2})$, using the inequality $\sin \theta > \frac{2}{\pi} \theta$, 
\begin{align*} \frac{d}{d\theta} {\log}G(\theta) & > \frac{2}{\theta} -\frac{2}{\pi+\theta} +\frac{1}{\pi-\theta}  -\frac{\cos \theta}{\sin \theta} \\
& >\left(2-\frac{\pi}{2} \right) \frac{1}{\theta} -\frac{2}{\pi+\theta} +\frac{1}{\pi-\theta} \\
&= \left(2-\frac{\pi}{2} \right) \frac{1}{\theta} +\frac{3\theta -\pi}{\pi^2 -\theta^2}.
\end{align*}
Obviously the last quantity is positive if $\theta \in (\frac{\pi}{3}, \frac{\pi}{2})$ because it is the sum of two positive terms.  Moreover, for $\theta \in (0, \frac{\pi}{3}]$, we  have the lower bound $\frac{\frac{3(4- \pi)}{2}-1}{\pi}$, which is larger than $0.09$.  

The above calculations show that $q_{\alpha, \beta}$ and $C_{\alpha, \beta}$ uniquely determine $\alpha$ and $\beta$. Since $b$ and $A$ are also known, the trapezoid is uniquely determined. 

We also note that in the case $\alpha =\frac{\pi}{2}$, the angle invariant $q$ determines $\beta$ uniquely, therefore the trapezoid can be determined again from the knowledge of $A$ and $b$.

\end{proof}

\section*{Acknowledgements} 
The first author is grateful to UC Irvine for its support. The second author is partially supported by NSF grant DMS-1547878.

\end{document}